\documentclass[letter, reqno, 12pt]{amsart}
\AtBeginDocument{\pagenumbering{arabic}}
\usepackage[usenames,dvipsnames]{color}
\usepackage{amsthm,amsfonts,amssymb,amsmath,amsxtra}
\usepackage{tikz}
\usepackage{verbatim}
\usepackage{amsmath}
\usepackage{amsxtra}
\usepackage{amscd}
\usepackage{leftindex}
\usepackage{amsthm}
\usepackage{amsfonts}
\usepackage{amssymb}
\usepackage{eucal}
\usepackage{stmaryrd}
\usepackage{romannum}
\usepackage{mathtools}
\usetikzlibrary{arrows}
\usepackage[all]{xy}
\SelectTips{cm}{}
\usepackage{xr-hyper}
\usepackage[colorlinks=
   citecolor=Black,
   linkcolor=Red,
  urlcolor=Blue]{hyperref}
\usepackage{verbatim}
\usepackage[T1]{fontenc}
\usepackage{leftidx}

\usepackage[margin=1.25in]{geometry}
\usepackage{mathrsfs}

\makeatletter
\newcommand{\extraNode}[6]%
{%
\dynkinPlaceRootRelativeTo{#1}{#2}{#3}{#4}{#5}
\dynkinIndefiniteSingleEdge{#1}{#2}
\dynkinRootMark{o}{#1}
\advance\dynkin@nodes by 1
\dynkinLabelRoot{#1}{#6} 
}%
\makeatother

\RequirePackage{xspace}
\RequirePackage{etoolbox}
\RequirePackage{varwidth}
\RequirePackage{enumitem}
\RequirePackage{tensor}
\RequirePackage{mathtools}
\RequirePackage{longtable}
\RequirePackage{multirow}

\def\<{\langle}
\def\>{\rangle}

\newcommand{\RNum}[1]{\uppercase\expandafter{\romannumeral #1\relax}}

\newcommand{\fkg}{\ensuremath{\mathfrak{g}}\xspace}

\newcommand{\fkp}{\ensuremath{\mathfrak{p}}\xspace}

\newcommand{\BF}{\ensuremath{\mathbb {F}}\xspace}
\newcommand{{\BG}}{\ensuremath{\mathbb {G}}\xspace}

\newcommand{{\BK}}{\ensuremath{\mathbb {K}}\xspace}

\newcommand{\BZ}{\ensuremath{\mathbb {Z}}\xspace}
\newcommand{\Bk}{\ensuremath{\mathbf {k}}\xspace}

\newcommand{\CB}{\ensuremath{\mathcal {B}}\xspace}
\newcommand{\CC}{\ensuremath{\mathcal {C}}\xspace}

\newcommand{\CO}{\ensuremath{\mathcal {O}}\xspace}
\newcommand{\CP}{\ensuremath{\mathcal {P}}\xspace}

\newcommand{\CS}{\ensuremath{\mathcal {S}}\xspace}

\newcommand{\Ad}{{\mathrm{Ad}}}
\newcommand{\ad}{{\mathrm{ad}}}

\newcommand{\Inv}{\mathrm{Inv}}

\newcommand{\id}{\ensuremath{\mathrm{id}}\xspace}

\newtheorem{theorem}{Theorem}
\newtheorem{proposition}[theorem]{Proposition}
\newtheorem{lemma}[theorem]{Lemma}

\newtheorem{corollary}[theorem]{Corollary}

\theoremstyle{definition}

\newtheorem{definition}[theorem]{Definition}

\newtheorem{remark}[theorem]{Remark}

\numberwithin{equation}{section}
\numberwithin{theorem}{section}

\theoremstyle{plain}

\setitemize[0]{leftmargin=*,itemsep=\the\smallskipamount}
\setenumerate[0]{leftmargin=*,itemsep=\the\smallskipamount}

\renewcommand{\to}{%
   \ifbool{@display}{\longrightarrow}{\rightarrow}%
   }
\let\shortmapsto\mapsto
\renewcommand{\mapsto}{%
   \ifbool{@display}{\longmapsto}{\shortmapsto}%
   }
\newlength{\olen}
\newlength{\ulen}
\newlength{\xlen}
\newcommand{\xra}[2][]{%
   \ifbool{@display}%
      {\settowidth{\olen}{$\overset{#2}{\longrightarrow}$}%
       \settowidth{\ulen}{$\underset{#1}{\longrightarrow}$}%
       \settowidth{\xlen}{$\xrightarrow[#1]{#2}$}%
       \ifdimgreater{\olen}{\xlen}%
          {\underset{#1}{\overset{#2}{\longrightarrow}}}%
          {\ifdimgreater{\ulen}{\xlen}%
             {\underset{#1}{\overset{#2}{\longrightarrow}}}
             {\xrightarrow[#1]{#2}}}}%
      {\xrightarrow[#1]{#2}}
   }
\makeatother
\newcommand{\xyra}[2][]{%
   \settowidth{\xlen}{$\xrightarrow[#1]{#2}$}%
   \ifbool{@display}%
      {\settowidth{\olen}{$\overset{#2}{\longrightarrow}$}%
       \settowidth{\ulen}{$\underset{#1}{\longrightarrow}$}%
       \ifdimgreater{\olen}{\xlen}%
          {\mathrel{\xymatrix@M=.12ex@C=3.2ex{\ar[r]^-{#2}_-{#1} &}}}%
          {\ifdimgreater{\ulen}{\xlen}%
             {\mathrel{\xymatrix@M=.12ex@C=3.2ex{\ar[r]^-{#2}_-{#1} &}}}
             {\mathrel{\xymatrix@M=.12ex@C=\the\xlen{\ar[r]^-{#2}_-{#1} &}}}}}%
      {\mathrel{\xymatrix@M=.12ex@C=\the\xlen{\ar[r]^-{#2}_-{#1} &}}}%
   }
\makeatletter
\newcommand{\xla}[2][]{%
   \ifbool{@display}%
      {\settowidth{\olen}{$\overset{#2}{\longleftarrow}$}%
       \settowidth{\ulen}{$\underset{#1}{\longleftarrow}$}%
       \settowidth{\xlen}{$\xleftarrow[#1]{#2}$}%
       \ifdimgreater{\olen}{\xlen}%
          {\underset{#1}{\overset{#2}{\longleftarrow}}}%
          {\ifdimgreater{\ulen}{\xlen}%
             {\underset{#1}{\overset{#2}{\longleftarrow}}}
             {\xleftarrow[#1]{#2}}}}%
      {\xleftarrow[#1]{#2}}
   }
\newcommand{\isoarrow}{%
   \ifbool{@display}{\overset{\sim}{\longrightarrow}}{\xrightarrow\sim}%
   }
\begin{document}

\title[On certain varieties attached to braids]{On certain varieties attached to braid representatives of a Weyl group element}

\thanks{}
\author[Chengze Duan]{Chengze Duan}
\address[C.Duan]{Department of Mathematics, University of Maryland, College Park, MD, 20740}
\email{cduan12@umd.edu}

\maketitle
\begin{abstract}
    Let $G$ be a connected reductive group over an algebraically closed field with Weyl group $W$. The analogy between Lusztig varieties and Deligne-Lusztig varieties associated to minimal length elements in elliptic conjugacy classes of $W$ was studied by He and Lusztig. In this paper, we study the parabolic version of above varieties associated to some good braid elements for any conjugacy class of $W$. This also leads to an alternative version of Lusztig's map in general case.
\end{abstract}

\section{Introduction}
\subsection{Overview}
Let $G$ be a connected reductive group over an algebraically closed field $\Bk$. Let $F:G\to G$ be the Frobenius map if $\Bk=\overline{\BF_q}$ for a finite field $\BF_q$. Fix a ($F$-stable) Borel subgroup $B_0$ of $G$ and let $W$ be the Weyl group of $G$. Let $\CB$ be the variety of Borel subgroups of $G$. Recall that the set of $G$-orbits on $\CB \times \CB$ for the diagonal $G$-action is indexed by $W$. Let $O_w$ be the $G$-orbit corresponding to $w\in W$. The length function on $W$ is then $l(w)=\dim O_w-\dim \CB$. For any $w\in W$, the associated Deligne-Lusztig variety and Lusztig variety are defined by
\begin{eqnarray*}
    X_w&=&\{B\in \CB \mid (B,F(B))\in O_w\};\\
    Y_w&=&\{(g,B)\in G\times \CB \mid (B,gBg^{-1})\in O_w\}.
\end{eqnarray*}
Deligne-Lusztig varieties were defined in \cite[Definition 1.4]{DL} and played an important role in the representation theory of finite groups of Lie type. Lusztig varieties were introduced first in \cite[\S 1]{Lu0}. The general form was defined in \cite[\S 2.4]{Lu01} and was used to define the character sheaves. This was further generalized to the twisted case in \cite[\S 28]{Lu02}.

In \cite[\S 4]{Lu1}, Lusztig explicitly constructed a map $\Phi$ from conjugacy classes of $W$ to unipotent conjugacy classes in $G$. In the same paper, he investigated the action of $G$ (resp. $G^F$) on Lusztig varieties (resp. Deligne-Lusztig varieties) corresponding to minimal length elements in any elliptic conjugacy class of $W$. Using these results, Lusztig gave an alternative definition of $\Phi$ for elliptic conjugacy classes of $W$. In \cite{Lu2} and \cite{HL}, the results on these varieties were generalized to the twisted case for disconnected groups. Moreover, the transversal slices constructed by He and Lusztig also play a role in these results. These results were later used in the study of He on affine Deligne-Lusztig varieties (see \cite[\S 4]{He1}).

In \cite[\S 3]{D}, the author constructed good position braid representatives for all conjugacy classes of $W$. In the elliptic case, these will give us minimal length elements. The author used these braid elements to construct transversal slices for all unipotent conjugacy classes. In this paper, we investigate the action of $G$ (resp. $G^F$) on Lusztig varieties (resp. Deligne-Lusztig varieties) corresponding to these elements. More specifically, we will consider the parabolic version of these varieties, which actually fit our braid elements better. Some of our results are based on the study of Digne and Michel on parabolic Deligne-Lusztig varieties (see \cite{DM}).

\subsection{Main results}
In this paper, we consider the following two cases.
\begin{itemize}
    \item case \RNum{1}: $G$ is the identity component of a disconnected reductive group $\tilde{G}$ with a fixed connected component $D$;
    \item case \RNum{2}: $\Bk=\overline{\BF_q}$ for some finite field $\BF_q$ and $F:G \to G$ is the Frobenius map.
\end{itemize}
In these two cases, the component $D$ and the Frobenius map $F$ define twists of $W$. This leads to a twisted Coxeter group $\tilde{W}$ in each case (see \S 2.3). Starting from a good position braid representative of a twisted conjugacy class in $\tilde{W}$, one can construct the corresponding parabolic (twisted) Lusztig variety $Y$ and parabolic Deligne-Lusztig variety $X$ together with their natural coverings $\tilde{Y}$ and $\tilde{X}$ (see \S 3.3 for precise definition). These varieties enable natural left actions of $G$ or $G^F$. 
\begin{theorem} (see also theorem \ref{3.7}) Let $w$ be the projection of a good position braid representative of a twisted conjugacy class in $\tilde{W}$. We have

(a) Any isotropy group of the $G$ (resp. $G^F$)-action on $\tilde{Y}$ (resp. $\tilde{X}$) is trivial;

(b) Any isotropy group of the $G$ (resp. $G^F$)-action on $Y$ (resp. $X$) is isomorphic to a subgroup of a reductive subgroup $L$ of $G$;

(c) The varieties $\tilde{Y}$ and $Y$ are affine. 
\end{theorem}

We refer to \S 3.3 for the definition of $L$. Now if we further require $D=G$, we have the following result.

\begin{proposition}(see also proposition \ref{4.8})
Let $w$ be the projection of a good position braid representative of a conjugacy class in $W$. We have

(a) The orbit space $G\backslash \tilde{Y}$ is isomorphic to the affine space $U^w$;

(b) The orbit space $G\backslash Y$ is isomorphic to $L\backslash U^w$;

(c) If $\CC=\Psi(\CO)$ for some $\CO\in [G_u]$, then $\CO$ is the unique $\CC$-small unipotent conjugacy class in $G$.
\end{proposition}

 We refer to definition \ref{4.1} for the definition of $\CC$-small. This generalizes the definition in \cite[\S 5.5]{Lu1}, which serves as an alternative version of Lusztig's map. The map $\Psi:[G_u]\to [W]$ is defined in \cite[Theorem 0.2]{Lu3}.

\section{Preliminaries}

\subsection{Lusztig's map $\Phi$}
Let $[G]$ be the set of conjugacy classes in $G$ and $[G_u]$ be the set of unipotent conjugacy classes in $G$. Let $(W,S)$ be its Weyl group and $[W]$ be the set of conjugacy classes of $W$.

For any subset $I\subset S$, denote the corresponding standard parabolic subgroup as $W_I$. Recall that a conjugacy class $\CC\in [W]$ is called \textit{elliptic} if for any $I\subsetneq S$, we have $\CC \cap W_I=\emptyset$. Let $[W_e]$ be the set of elliptic conjugacy classes of $W$. Recall the following result of Lusztig.

\begin{theorem}\cite[Theorem 0.4, Proposition 0.6]{Lu1}
    Lusztig's map $\Phi$ is a surjective map from $[W]$ to $[G_u]$. Moreover, the resctriction of $\Phi$ on $[W_e]$ is injective.
\end{theorem}

The first definition of this map is based on the excellent decomposition of certain elliptic elements in $W$ (see \cite[\S 2]{Lu1}). An explicit description of the restriction map of $\Phi$ to $[W_e]$ is given in \cite[\S 4]{Lu1} (See also \cite[\S 2.2]{D} for an explicit description in non-elliptic case).

\subsection{$\CC$-small classes}
In this section we recall some results in \cite[\S 5]{Lu1}. Recall there is a natural stratification
$Y_w=\underset{\CO\in [G]}{\bigsqcup} Y_w^{\CO}$, where
\begin{equation*}
    Y_w^{\CO}=\{(g,B)\in Y_w \mid g\in \CO\}.
\end{equation*}
We know that $G$ and $G_{\ad}=G\slash Z(G)$ act on $Y_w$ and $Y_w^{\CO}$ by
\begin{eqnarray*}
    x&:& (g,B)\mapsto (xgx^{-1},xBx^{-1});\\
    xZ(G)&:& (g,B)\mapsto (xgx^{-1},xBx^{-1}),
\end{eqnarray*}
and $Y_w^\CO$ is a union of $G_{\ad}$-orbits in $Y_w$. Now let $w$ be a minimal length element in an elliptic conjugacy class $\CC\in [W_e]$. Lusztig showed that $\dim Y_w^\CO\geq \dim G_\ad$ if $Y_w^\CO$ is nonempty (see \cite[Theorem 5.2]{Lu1}).
\begin{definition}\cite[\S 5.5]{Lu1}
    Let $\CC$ be an elliptic conjugacy class of $W$. A conjugacy class $\CO\in [G]$ is called \textit{$\CC$-small} if $\dim Y_w^\CO=\dim G_\ad$ and $Y_w^\CO \neq \emptyset$.
\end{definition}
Lusztig proved that there is a unique $\CC$-small unipotent conjugacy class in $G$, which is exactly $\Phi(\CC)$ (see \cite[\S 5.8]{Lu1}).

\subsection{Deligne-Lusztig varieties and (twisted) Lusztig varieties}
In this section we recall the results in \cite{Lu2}. First recall a twist $\delta$ of $W$ is a group automorphism $W\to W$ preserving the set of simple reflections. Two elements $w,w'\in W$ are called $\delta$-conjugate if there exists $x\in W$ such that $w'=\delta(x)wx^{-1}$. Let $\tilde{W}=\langle \delta \rangle \ltimes W$ be the twisted Weyl group. In \cite[Remark 2.1]{GKP}, it is proved that $w\mapsto \delta w$ defines a bijection between the $\delta$-conjugacy classes of $W$ and the $W$-conjugacy classes of $\tilde{W}$ contained in the coset $\delta W$. A $\delta$-conjugacy class $\CC^\delta$ is called \textit{$\delta$-elliptic} if $\CC^\delta\cap W_I=\emptyset$ for any $\delta$-stable subset $I\subsetneq S$.

Recall our two cases are as follows.
\begin{itemize}
    \item case \RNum{1}: $G$ is the identity component of a disconnected reductive group $\tilde{G}$ with a fixed connected component $D$;
    \item case \RNum{2}: $\Bk=\overline{\BF_q}$ for some finite field $\BF_q$ and $F:G \to G$ is the Frobenius map.
\end{itemize}
In the first case, the component $D$ defines a twist $\delta_D:W\to W$ satisfying $gO_wg^{-1}=O_{\delta_D(w)}$ for $g\in D$. In the second case, the Frobenius map $F$ defines a twist $\delta_F:W \to W$  satisfying $F(O_w)=O_{\delta_F(w)}$. For any $w\in W$, the (twisted) Lusztig variety are defined as
\begin{equation*}
    Y_w^D=\{(g,B)\in D\times \CB \mid (B,gBg^{-1})\in O_w\}.
\end{equation*}
Similar as before, we have $G$ (resp. $G^F$) acts on $Y_w^D$ (resp. $X_w$) by $x\cdot (g,B)=(xgx^{-1},xBx^{-1})$ (resp. $x\cdot B=xBx^{-1}$). 

In \cite[\S 1.11]{DL} and \cite[\S 0.2]{Lu2}, we see that $X_w$ and $Y_w^D$ admit natural finite coverings. Let $T$ be a maximal torus of a fixed ($F$-stable) Borel subgroup $B_0$. Let $U$ be the unipotent radical of $B_0$ and $U^-$ be opposite to $U$. Choose $d\in D$ be such that $dTd^{-1}=T$ and $dB_0d^{-1}=B_0$. As in \cite[\S 0.2]{Lu2}, choose representatives $\dot{w}$ for all $w\in W$ in $N_G(T)$ such that $\dot{w}=\dot{w_1}\dot{w_2}$ if $w=w_1w_2$ is reduced, and $^{\delta_D}\dot{w}=d\dot{w}d^{-1}$ or $^{\delta_F}\dot{w}=F(\dot{w})$ respectively. Here we denote the representative of $\delta(w)$ as $^\delta\dot{w}$. Let $^wU=U\cap \dot{w}U\dot{w}^{-1}$ be generated by root subgroups corresponding to $R^+\setminus \Inv(w^{-1})$ and $U^w=U\cap \dot{w}U^-\dot{w}^{-1}$ be generated by root subgroups corresponding to $\Inv(w^{-1})$. Set $T^w=\{t\in T\mid \dot{w}^{-1}t\dot{w}=dtd^{-1}\}$ in case \RNum{1} or $T^w=\{t\in T\mid \dot{w}^{-1}t\dot{w}=F(t)\}$ in case \RNum{2}. Define
\begin{eqnarray*}
    \tilde{X}_w &=&\{g ^wU\in G\slash ^wU \mid g^{-1}F(g)\in \dot{w}U\};\\ 
    \tilde{Y}_w^D &=&\{(g,g'^wU)\in D\times G\slash ^wU \mid g'^{-1}gg'\in \dot{w}Ud\},
\end{eqnarray*}
and $G$ (resp. $G^F$) acts on $\tilde{Y}_w^D$ (resp. $\tilde{X}_w$) by
\begin{eqnarray*}
    x&:& (g,g'^wU)\mapsto (xgx^{-1},xg'^wU);\\
    x&:& g'^wU\mapsto xg'^wU.
\end{eqnarray*}
We know these are natural coverings of $X_w$ and $Y_w^D$ via $(g,g'^wU)\mapsto (g,g'B_0g'^{-1})$ and $g'^wU\mapsto g'B_0g'^{-1}$. For convenience, we denote these two covering maps both as $\pi_w$. 

\begin{theorem}\cite[Theorem 0.3, Theorem 0.4]{Lu2} \cite[\S 4.2 \S 4.3]{HL}
    Let $w$ be a minimal length element in a $\delta$-elliptic conjugacy class of $W$. In case \RNum{1} we further assume $G$ is semisimple. Then we have

    (a) Any isotropy group of the $G$ (resp. $G^F$)-action on $\tilde{Y}_w^D$ (resp. $\tilde{X}_w$) is trivial;

    (b) Any isotropy group of the $G$ (resp. $G^F$)-action on $Y_w^D$ (resp. $X_w$) is isomorphic to a subgroup of $T^w$, which is a finite abelian group;

    (c) The varieties $Y_w^D$ and $\tilde{Y}_w^D$ are affine;

    (d) The orbit space $G\backslash \tilde{Y}_w^D$ is isomorphic to $U^{\delta_D(w)}$ and the orbit space $G\backslash Y_w^D$ is isomorphic to $T^w \backslash U^{\delta_D(w)}$;

    (e) The orbit space $G\backslash \tilde{X}_w^D$ is quasi-isomorphic to $U^{\delta_F(w)}$ and the orbit space $G\backslash X_w^D$ is quasi-isomorphic to $T^w \backslash U^{\delta_F(w)}$.
\end{theorem}
The analogue of part (c) for $X_w$ and $\tilde{X}_w$ was first proved in \cite[Corollary 1.12]{DL} if $\text{char}(\Bk)$ is good. For arbitrary $\text{char}(\Bk)$, it was proved in \cite[\S 5]{OR} for split classical groups and \cite[Theorem 1.3]{He} in general. It can also be deduced from \cite[Theorem B]{BR}. This theorem is also one of the intrinsic reason of the construction of transversal slices in \cite{HL}.

\section{Non-elliptic case}
In this section, we investigate the above varieties without the assumption on ellipticity and replace minimal length elements by some other elements. We keep our settings in \S 2.3.

\subsection{Good position braid representatives}
We first recall some results in \cite[\S 3]{D}. Let $\tilde{W}=\langle \delta \rangle \ltimes W$ be a twisted finite Coxeter group where $\delta$ is a twist of $W$. Let $R$ be the roots of $W$. Let $B^+(\tilde{W})$ be the braid monoid associated to $\tilde{W}$ and $\pi_{\tilde{W}}$ be the natural projection. For any $w\in W$, denote the natural embedding of $w$ in $B^+(\tilde{W})$ as $\underline{w}$. We have the following result, which is a special case of \cite[Proposition 3.2]{D} together with \cite[Lemma 5.1]{HN}.
\begin{proposition}\label{3.1}
    Let $\tilde{\CC}$ be a $W$-conjugacy class of $\tilde{W}$. There exists a positive integer $d$ and a good position braid representative $\tilde{b}=\sigma b$ of $\tilde{\CC}$ with $\sigma\in \langle \delta \rangle$ and $b\in B^+(W)$ such that $\pi_{\tilde{W}}(\tilde{b})\in \tilde{\CC}$ and 
    \begin{equation*}
        \tilde{b}^d=\sigma^d \underline{w_0}\cdot b',
    \end{equation*}
    where $b'\in B^+(W)$ and $w_0$ is the longest element in $W$.
\end{proposition}

\begin{remark}
    If $\tilde{\CC}$ is elliptic, then any good position braid representative of $\tilde{\CC}$ is of the form $\underline{w}$ for some minimal length element $w\in \tilde{\CC}$.
\end{remark}

\subsection{Deligne-Lusztig varieties and Lusztig varieties}
From now on, we always set $\delta=\delta_D$ in case \RNum{1} and $\delta=\delta_F$ in case \RNum{2}. Let $\CC^\delta$ be any $\delta$-conjugacy class of $W$ and it corresponds to a $W$-conjugacy class $\tilde{\CC}$ of $\tilde{W}$ in $\delta{W}$. Take a braid element $\tilde{b}$ in Proposition \ref{3.1} for $\tilde{\CC}$. By assumption we have $\tilde{b}=\delta b$ for some $b\in B^+(W)$. Then we have
\begin{equation*}
    b\delta(b)\delta^2(b)\cdots \delta^{d-1}(b)=(\tilde{b})^d=\delta^d\underline{w_0}\cdot b',
\end{equation*}
for some $b'\in B^+(W)$. Now we take $w=\pi_W(b)$ be the projection of $b$ to $W$. Let $R^{\tilde{w}}$ be the fixed root of $\tilde{w}=\delta w$. One can deduce the following corollary from \cite[Proposition 3.2]{D}.
\begin{corollary}\label{3.3}
   Let $\tilde{w}=\delta w$ be as above. There exists $d$ such that
   \begin{equation*}
    (\delta \underline{w})^d=\tilde{w}^d=\delta^d\underline{w_0w'}\cdot b',
\end{equation*}
where $w'$ is the longest element of the standard parabolic subgroup $W'<W$ corresponding to $R^{\tilde{w}}$ and $b'\in B^+(W)$.
\end{corollary}

\begin{remark}
    One should be careful that the element $w$ above depends on $\delta$. In other words, different choice of $\delta$, such as $\delta_D$ and $\delta_F$, may give different $w$.
\end{remark}

Let $U_{\tilde{w}}$ be generated by root subgroups corresponding to positive roots in $R_{\tilde{w}}$. Then we have the following result on the Lusztig variety attached to $w$ in case \RNum{1}.
\begin{proposition}\label{3.5}
    Let $\tilde{w}=\delta_D w$ be as above. Then any isotropy group of the $G$-action on $Y_w^D$ is isomorphic to a subgroup of $T^wU_{\tilde{w}}$.
\end{proposition}
\begin{proof}
    We follow the idea of \cite[Theorem 5.2]{Lu1}. Let $(g,B)$ be arbitrary in $Y_w^D$ and $G_{g,B}<G$ be its isotropy group. By our choice of $w$, there exists $k\geq 1$ such that $\delta_D^k=1$ and
    \begin{equation*}
        w\delta_D(w)\delta_D^2(w)\cdots \delta_D^{k-1}(w)=\underline{w_0w'}\cdot b',
    \end{equation*} 
    by corollary \ref{3.3}. Let $s_1s_2\ldots s_l$ be a reduced expression of $w$ and $s_1's_2'\ldots s_r'$ be a reduced expression of $w_0w'$. We also write $b'=s_1''s_2''\ldots s_m''$. We first construct a collection $\{\CS^i\}_{1\leq i\leq n}$ of sequences of simple reflections of $W$ satisfying the following conditions:
        \begin{itemize}
        \item $\CS^1=(s_1,s_2,\ldots,s_l,\delta_D(s_1),\delta_D(s_2),\ldots,\delta_D(s_l),\ldots,\delta_D^{k-1}(s_1),\ldots,\delta_D^{k-1}(s_l))$ with $kl$ terms;
        \item $\CS^{i+1}$ is obtained from $\CS^i$ by replacing a subsequence $(s^i_{a+1},\ldots,s^i_{a+r})$ of the form $(s,t,s,t,\ldots)$ by $(t,s,t,s,\ldots)$ where $s\neq t$ and $st$ is of order $r$ in $W$. In other words, the sequence $\CS^{i+1}$ is obtained from $\CS^i$ by applying a braid relation;
        \item $\CS^n=(s_1',s_2',\ldots,s_h',s_1'',\ldots,s_m'')$ also with $kl$ terms.
    \end{itemize}
Based on these sequences, we construct a collection $\{B_*^i\}_{1\leq i\leq n}$  of sequences of Borel subgroups of $G$ satisfying the following conditions:
        \begin{itemize}
        \item $B_*^1=(B_0^1,B_1^1,\ldots, B_{kl}^1)$, where $B_{il}^1=g^iBg^{-i}$ for $0\leq i\leq k$ and $(B_{il+j-1}^1,B_{il+j}^1)\in O_{\delta_D^i(s_j)}$ for $0\leq i\leq k-1$, $1\leq j\leq l$;
        \item $B_*^{i+1}$ is obtained from $B_*^i$ by replacing $(B^i_a,\ldots B^i_{a+r})$ by $(B^{i+1}_a,\ldots B^{i+1}_{a+r})$. Here $B^{i+1}_a=B^i_a$ and $B^{i+1}_{a+r}=B^i_{a+r}$ and $(B^{i+1}_{a+j-1},B^{i+1}_{a+j})\in O_{s^{i+1}_{a+j}}$ for $1\leq j\leq r$. In other words, we have $(B_a^{i+1},B_{a+1}^{i+1})\in O_t$ while $(B_a^{i},B_{a+1}^{i})\in O_s$ and so on;
        \item $B_*^n=(B_0^n,\ldots,B_h^n,B_{h+1}^n,\ldots,B_{h+m}^n)$, where $(B_{i-1}^n,B_{i}^n)\in O_{s_i'}$ for $1\leq i\leq h$ and $B_{h+i-1}^n,B_{h+i}^n)\in O_{s_i''}$ for $1\leq i\leq m$.
    \end{itemize}
  We first consider $B_*^1$. Since $s_1s_2\ldots s_l$ is a reduced expression of $w$ and $(B_0^1,B_l^1)=(B,gBg^{-1})\in O_w$, we have $B_1^1,\ldots,B_{l-1}^1$ are uniquely determined. Similarly, one can see that the sequence $B_*^1$ is uniquely determined. Now notice that the two expressions of a braid relation are both reduced, we have $B^i_*$ are all uniquely determined.

  Now let $x\in G_{g,B}$. By definition we have $xgx^{-1}=g$ and $xBx^{-1}=B$. This tells us $xB_{il}^1x^{-1}=B_{il}^1$ for any $0\leq i\leq k$. Now since $B^1_*$ is uniquely determined, by applying the action of $x$ to it we have $xB_i^1x^{-1}=B_i^1$ for any $0\leq i\leq kl$. Now assume that $xB_j^ix^{-1}=B_j^i$ for $0\leq j\leq kl$. Since $B_*^{i+1}$ is uniquely determined by applying a braid relation to $B_*^i$, we have $xB_j^{i+1}x^{-1}=B_j^{i+1}$ for $0\leq j\leq kl$. Therefore, any Borel subgroup in any $B_*^i$ is fixed by $x$. In particular, we have $xB_h^nx^{-1}=B_h^n$. Since $(B,B_h^n)\in O_{w_0w'}$, we have $x\in B\cup B_h^n$, which is a maximal torus inside $B$ and the unipotent part in $B$ corresponding to $R^{\tilde{w}}$.

  By above, we have $G_{g,B}$ is contained in $T_BU_{\tilde{w}}$ for a maximal torus $T_B\subset B\cap gBg^{-1}$ and $U_{\tilde{w}}\subset U_B$ corresponding to $R_{\tilde{w}}$ with respect to $U_B$. As in \cite[\S 0.2]{Lu2}, let $d_B\in D$ be such that $d_B B d_B^{-1}=B$ and $d_B T_B d_B^{-1}=T_B$. One can write $g=u_1hd_Bu_2$ where $u_1,u_2\in U_B$ and $h\in N_G(T_B)$. Here $hd_B$ is uniquely determined. Now let $x=tu=u't\in G_{g,B}$ where $t\in T_B$ and $u,u'\in U_{\tilde{w}}$. We have $u_1hd_Bu_2=g=xgx^{-1}=xu_1x^{-1}xhd_Bx^{-1}xu_2x^{-1}=xu_1x^{-1}u'thd_Bt^{-1}u'^{-1}xu_2x^{-1}$. Notice that $xu_1x^{-1}u'=tuu_1u^{-1}t^{-1}u'=tuu_1u^{-1}ut^{-1}=tuu_1t^{-1}\in U_B$. Similarly we have $u'^{-1}xu_2x^{-1}\in U_B$. By uniqueness, we have $thd_Bt^{-1}=hd_B$. Finally, since $(B,gBg^{-1})\in O_w$, we have $(B,hBh^{-1})\in O_w$. This together with above tell that $t\dot{w}d_B t^{-1}=\dot{w}d_B$ for $\dot{w}\in N_G(T_B)\slash T_B$. Thus $t\in \{t\in T_B \mid \dot{w}^{-1}t\dot{w}=d_Btd_B^{-1}\}\simeq T^w$. Therefore, our statement is proved.
\end{proof}

As for Deligne-Lusztig varieties, we can similarly prove the following result.
\begin{corollary}
    Let $\tilde{w}=\delta_F w$ be as above. Then any isotropy group of the $G^F$-action on $X_w$ is isomorphic to a subgroup of $TU_{\tilde{w}}$.
\end{corollary}
However, it is not clear if the torus part lie in $T^w$ or not.

\subsection{Parabolic version}
In this section, we consider the parabolic version of above varieties, which enable us to construct nice coverings.

Let $I\subset S$ and $W_I$ be the corresponding standard parabolic subgroup. We denote the set of minimal coset representative in $\tilde{W}\slash W_I$ (resp. $W_I\backslash \tilde{W}$) as $\tilde{W}^I$ (resp. $^I\tilde{W}$). Set $^I\tilde{W}^I$ as $^I {\tilde{W}}\cap \tilde{W}^I$. 

Let $\CC^\delta,\tilde{\CC}$ and $\tilde{w}=\delta w$ be as in \S 3.2. By corollary \ref{3.3}, we set $I_{\tilde{w}}\subset S$ be the simple reflections of the standard parabolic subgroup $W'<W$ corresponding to $R^{\tilde{w}}$. As in \S 2.3, we fix a ($F$-stable) Borel subgroup $B_0$ and a maximal torus $T\subset B_0$ and identify $W$ with $N_G(T)\slash T$. For any $J\subset S$, let $P_J=B_0W_JB_0$ be a standard parabolic subgroup of $G$ with and $\CP_J$ be the associated partial flag variety. Denote the Levi subgroup of $P_J$ by $L_J$ and the unipotent radical by $U_J$. Let $U_{(R^{\tilde{w}})^-}$ be generated by all roots subgroups corresponding to $(R^{\tilde{w}})^-$, the negative roots in $R^{\tilde{w}}$. Then $P_{I_{\tilde{w}}}=B_0W_{I_{\tilde{w}}}B_0=B_0U_{(R^{\tilde{w}})^-}$. Now in case \RNum{1}, again choose $d\in D$ such that $dB_0d^{-1}=B_0$ and $dTd^{-1}=T$. Then we have $dP_{\delta_D^{-1}(I_{\tilde{w}})}d^{-1}=B_0W_{I_{\tilde{w}}}B_0=P_{I_{\tilde{w}}}$. In case \RNum{2}, since $B_0$ is $F$-stable, we also have $F(P_{\delta_F^{-1}(I_{\tilde{w}})})=P_{I_{\tilde{w}}}$. As in \S 2.3, for any $J\subset S$, we set $L_{J}^w=\{l\in L_J\mid \dot{w}^{-1}l\dot{w}=dld^{-1}\}$ in case \RNum{1} and $L_{J}^w=\{l\in L_J \mid \dot{w}^{-1}l\dot{w}=F(l)\}$ in case \RNum{2}.

Recall that given any twist $\delta$ of $W$, there is a natural diagonal action of $G$ on $\CP_{\delta^{-1}(I_{\tilde{w}})} \times \CP_{I_{\tilde{w}}}$. The set of $G$-orbits of this action is index by $^{\delta^{-1}(I_{\tilde{w}})}W^{I_{\tilde{w}}}$. Denote the $G$-orbit corresponding to a minimal double coset representative $u\in {}^{\delta^{-1}(I_{\tilde{w}})}W^{I_{\tilde{w}}}$ as $O_{\delta^{-1}(I_{\tilde{w}}),I_{\tilde{w}}}(u)$. Now for any $u\in {}^{\delta^{-1}(I_{\tilde{w}})}W^{I_{\tilde{w}}}$ (recall here $\delta=\delta_D$ in case \RNum{1} and $\delta=\delta_F$ in case \RNum{2}), recall the associated generalized Deligne-Lusztig varieties and generalized Lusztig varieties are defined as
\begin{eqnarray*}
    X_{\delta_F^{-1}(I_{\tilde{w}}),I_{\tilde{w}}}(u)&=&\{P\in \CP_{\delta_F^{-1}(I_{\tilde{w}})}\mid (P,F(P))\in O_{\delta_F^{-1}(I_{\tilde{w}}),I_{\tilde{w}}}(u)\};\\
Y_{\delta_D^{-1}(I_{\tilde{w}}),I_{\tilde{w}}}^D(u)&=&\{(g,P)\in D\times \CP_{\delta_D^{-1}(I_{\tilde{w}})}\mid (P,gPg^{-1})\in O_{\delta_D^{-1}(I_{\tilde{w}}),I_{\tilde{w}}}(u)\}.
\end{eqnarray*}
There is a natural left $G$ (resp. $G^F$)-action on $Y_{\delta_D^{-1}(I_{\tilde{w}}),I_{\tilde{w}}}^D(u)$ (resp. $X_{\delta_F^{-1}(I_{\tilde{w}}),I_{\tilde{w}}}(u)$) given by $x:(g,P)\mapsto (xgx^{-1},xPx^{-1})$ (resp. $x:P\mapsto xPx^{-1}$).

We know $w$ is the element in $W$ such that $\tilde{w}=\delta w\in \tilde{W}$ is the natural projection of a good position braid $\tilde{b}$ of a conjugacy class $\tilde{\CC}$. By our contruction of $\tilde{b}$ in \cite[\S 3.2]{D} (see also \cite[Proposition 2.2]{HN}), we have $\tilde{w}$ is a minimal double coset representative in $W_{I_{\tilde{w}}}\backslash \tilde{W} \slash W_{I_{\tilde{w}}}$ and thus $w\in {}^{\delta^{-1}(I_{\tilde{w}})}W^{I_{\tilde{w}}}$. From now on, we focus on the case where $u=w$. In this case, these generalized Lusztig varieties are distinguished in the sense of \cite{Ha}. In the words of \cite[\S 5]{DM}, the triple $(\delta^{-1}(I_{\tilde{w}}),w,I_{\tilde{w}})$ is a morphism in the ribbon category $B^+(I)$. 

Now we have natural coverings of $X_{\delta_F^{-1}(I_{\tilde{w}}),I_{\tilde{w}}}(w)$ and $Y_{\delta_D^{-1}(I_{\tilde{w}}),I_{\tilde{w}}}^D(w)$ defined as
\begin{eqnarray*}
    \tilde{X}_{\delta_F^{-1}(I_{\tilde{w}}),I_{\tilde{w}}}(w)&=&\{g'U_{\delta_F^{-1}(I_{\tilde{w}})}\in G\slash U_{\delta_F^{-1}(I_{\tilde{w}})} \mid g'^{-1}F(g')\in U_{\delta_F^{-1}(I_{\tilde{w}})}\dot{w} U_{I_{\tilde{w}}}\}; \\
    \tilde{Y}^D_{\delta_D^{-1}(I_{\tilde{w}}),I_{\tilde{w}}}(w)&=&\{(g,g'U_{\delta_D^{-1}(I_{\tilde{w}})})\in D\times G\slash U_{\delta_D^{-1}(I_{\tilde{w}})} \mid g'^{-1}gg'\in U_{\delta_D^{-1}(I_{\tilde{w}})}\dot{w} U_{I_{\tilde{w}}}d\}.
\end{eqnarray*}
The first covering $ \tilde{X}_{\delta_F^{-1}(I_{\tilde{w}}),I_{\tilde{w}}}(w)$ was first defined in \cite[Definition 2.3.11]{DMR} and was generalized in \cite[\S 7.3]{DM}. Again there is a natural left action of $G$ (resp. $G^F$) on $\tilde{Y}^D_{\delta_D^{-1}(I_{\tilde{w}}),I_{\tilde{w}}}(w)$ (resp. $\tilde{X}_{\delta_F^{-1}(I_{\tilde{w}}),I_{\tilde{w}}}(w)$) given by $x:(g,g'U_{\delta_D^{-1}(I_{\tilde{w}})})\mapsto (xgx^{-1},xg'U_{\delta_D^{-1}(I_{\tilde{w}})})$ (resp. $x: g'U_{\delta_F^{-1}(I_{\tilde{w}})}\mapsto xg'U_{\delta_F^{-1}(I_{\tilde{w}})}$). Meanwhile, there is also a natural left action of $L_{\delta^{-1}(I_{\tilde{w}})}^w$ on $\tilde{Y}^D_{\delta_D^{-1}(I_{\tilde{w}}),I_{\tilde{w}}}(w)$ (resp. $\tilde{X}_{\delta_F^{-1}(I_{\tilde{w}}),I_{\tilde{w}}}(w)$) given by $l:(g,g'U_{\delta_D^{-1}(I_{\tilde{w}})}) \mapsto (g,g'l^{-1}U_{\delta_D^{-1}(I_{\tilde{w}})})$ (resp. $g'U_{\delta_F^{-1}(I_{\tilde{w}})}\mapsto g'l^{-1}U_{\delta_F^{-1}(I_{\tilde{w}})}$). Moreover, this action commutes with the $G$ (resp. $G^F$)-action. 

The covering maps are defined as $\pi_{\delta_D,w}(g,g'U_{\delta_D^{-1}(I_{\tilde{w}})})=(g,g'P_{\delta_D^{-1}(I_{\tilde{w}})}g'^{-1})$ in case \RNum{1} and $\pi_{\delta_F,w}(g'U_{\delta_F^{-1}(I_{\tilde{w}})})\mapsto g'P_{\delta_F^{-1}(I_{\tilde{w}})}g'^{-1}$ in case \RNum{2}. Clearly these covering maps is compatible with the left action of $L_{\delta^{-1}(I_{\tilde{w}})}^w$ where it acts on $X_{\delta_F^{-1}(I_{\tilde{w}}),I_{\tilde{w}}}(w)$ and $Y^D_{\delta_F^{-1}(I_{\tilde{w}}),I_{\tilde{w}}}(w)$ trivially. In both cases, we know the covering map is a principal $L_{\delta^{-1}(I_{\tilde{w}})}^w$-bundle. 

Now we can prove the first main result. The analogue of (c) for parabolic Deligne-Lusztig varieties was proved in \cite[\S 7.5]{DM}.

\begin{theorem}\label{3.7}
    Let $\CC^{\delta}$ be a $\delta$-conjugacy class of $W$ and $\tilde{w}=\delta w$ be the projection of a good position braid of $\CC^{\delta}$. Then we have

    (a) Any isotropy group of the $G$ (resp. $G^F$)-action on variety $\tilde{Y}^D_{\delta_D^{-1}(I_{\tilde{w}}),I_{\tilde{w}}}(w)$ (resp. $\tilde{X}_{\delta_F^{-1}(I_{\tilde{w}}),I_{\tilde{w}}}(w)$) is trivial;

    (b) Any isotropy group of the $G$ (resp. $G^F$)-action on variety $Y^D_{\delta_D^{-1}(I_{\tilde{w}}),I_{\tilde{w}}}(w)$ (resp. $X_{\delta_F^{-1}(I_{\tilde{w}}),I_{\tilde{w}}}(w)$) is isomorphic to a subgroup of $L_{I_{\tilde{w}}}^w$;

    (c) The varieties $\tilde{Y}^D_{\delta_D^{-1}(I_{\tilde{w}}),I_{\tilde{w}}}(w)$ and $Y^D_{\delta_D^{-1}(I_{\tilde{w}}),I_{\tilde{w}}}(w)$ are affine;
\end{theorem}

\begin{proof}
    (a) Let $(g,g'U_{\delta_D^{-1}(I_{\tilde{w}})})$ be arbitrary in $\tilde{Y}^D_{\delta_D^{-1}(I_{\tilde{w}}),I_{\tilde{w}}}(w)$ and $G_{g,g'}$ be its isotropy group. Let $x\in G_{g,g'}$, we have $xg'U_{\delta_D^{-1}(I_{\tilde{w}})}=g'U_{\delta_D^{-1}(I_{\tilde{w}})}$. Thus $g'^{-1}xg'\in U_{\delta_D^{-1}(I_{\tilde{w}})}$ and $x$ is unipotent. Moreover, we have $x\in g'U_{\delta_D^{-1}(I_{\tilde{w}})}g'^{-1}\subset g'B_0g'^{-1}$.
    
    Again there exists $k\geq 1$ such that $\delta_D^k=1$ and 
    \begin{equation*}
        w\delta_D(w)\delta_D^2(w)\cdots \delta_D^{k-1}(w)=\underline{w_0w'}\cdot b'.
    \end{equation*}
Choose the reduced expressions of $w$, $w_0w'$ and $b$ as in the proof of proposition \ref{3.5}. Then one can again construct a collection $\{\CS^i\}_{1\leq i\leq n}$ of sequences of simple reflections as there. By applying $\delta_D^{-1}$ to this collection we get a new collection $\{\CS'^i\}_{1\leq i\leq n}$. Now we set $B_{il}^1=g^ig'B_0g'^{-1}g^{-i}$ for $0\leq i\leq k$.
One can similarly construct a collection $\{B_*^i\}_{1\leq i\leq n}$ as before based on the new collection $\{\CS'^i\}_{1\leq i\leq n}$. By our assumption, we have $xg'B_0g'^{-1}x^{-1}=g'B_0g'^{-1}$ and $xgx^{-1}$. This again implies that $xB_{i}^mx^{-1}=B_i^m$ for $0\leq i\leq kl$ and $1\leq m \leq n$. In particular, consider $x\in B_0^n\cap B_h^n$. We know $x$ is unipotent and $x$ actually lies in $g'U_{\delta_D^{-1}(I_{\tilde{w}})}g'^{-1}$, which corresponds to the roots in $\delta_D^{-1}(R^+\setminus R^{\tilde{w}})$ with respect to $g'B_0g'^{-1}=B_0^n$. Since $(B_0^n,B_h^n)\in O_{\delta_D^{-1}(w_0w')}$ by our assumption, we must have $x=1$ and the isotropy group $G_{g,g'}$ is trivial. One can prove this for $\tilde{X}_{\delta_F^{-1}(I_{\tilde{w}}),I_{\tilde{w}}}(w)$ in exactly the same way as $x\in G^F$.

(b) Again we only prove for $Y^D_{\delta_D^{-1}(I_{\tilde{w}}),I_{\tilde{w}}}(w)$ and the other case is the same. Let $(g,P)\in Y^D_{\delta_D^{-1}(I_{\tilde{w}}),I_{\tilde{w}}}(w)$ and $G_{g,P}$ be the associated isotropy group. One can find $g'$ such that $\pi_{\delta_D,w}(g,g'U_{\delta_D^{-1}(I_{\tilde{w}})})=(g,P)$. Let $x\in G_{g,P}$, we then have $\pi_{\delta_D,w}(g,xg'U_{\delta_D^{-1}(I_{\tilde{w}})})=(xgx^{-1},xPx^{-1})=\pi_{\delta_D,w}(g,g'U_{\delta_D^{-1}(I_{\tilde{w}})})$. Recall that $\pi_{\delta_D,w}$ is a principal $L_{\delta_D^{-1}(I_{\tilde{w}})}^w$-bundle, we have find a unique $l\in L_{\delta_D^{-1}(I_{\tilde{w}})}^w$ such that $xg'U_{\delta_D^{-1}(I_{\tilde{w}})}=g'l^{-1}U_{\delta_D^{-1}(I_{\tilde{w}})}$. This defines a homomorphism $x\mapsto l$ from $G_{g,P}$ to $L_{\delta_D^{-1}(I_{\tilde{w}})}^w$. Moreover, the kernel of this maps lies in an isotropy group of $G$-action on $\tilde{Y}^D_{\delta_D^{-1}(I_{\tilde{w}}),I_{\tilde{w}}}(w)$ which is trivial. We have $G_{g,P}$ is isomorphic to a subgroup of $L_{\delta_D^{-1}(I_{\tilde{w}})}^w$ and thus isomorphic to a subgroup of $L_{I_{\tilde{w}}}^w$.

(c) This proof is motivated by the proof of Digne and Michel. We first prove that $\tilde{Y}^D_{\delta_D^{-1}(I_{\tilde{w}}),I_{\tilde{w}}}(w)$ is affine. As in \cite[Definition 7.10]{DM}, if $u$ is a minimal coset representative in $^IW$ for some $I\subset S$ and $u^{-1}Iu=J$ for some $J\subset S$, then we call $(I,u,J)$ an admissible triple and define
\begin{equation*}
    \tilde{O}(I,u)=\{(gU_I,g'U_J)\in G\slash U_I \times G\slash U_J \mid g^{-1}g'\in U_I\dot{u}U_J\}.
\end{equation*}
Moreover, let $b\in B^+(W)$ be arbitrary such that the longest left divisor of $b$ in $B^+(W_I)$ is trivial and $\pi_W(b)^{-1}I\pi_W(b)=J$ for some $J\subset S$. Suppose that there is a decomposition $b=w_1w_2\cdots w_m$ with a sequence $I=I_0,I_1,\ldots,I_m=J$ such that $(I_{i-1},w_{i},I_{i})$ is an admissible triple, we define 
\begin{equation*}
    \tilde{O}(I,b)=\tilde{O}(I,w_1)\times_{G\slash U_{I_1}} \cdots \times_{G\slash U_{I_{m-1}}} \tilde{O}(I_{m-1},w_m).
\end{equation*}
By \cite[\S 7.3]{DM}, this is well-defined. In other words, if we choose another suitable decomposition of $b$, the resulting variety will be isomorphic to this variety.

Now we choose $b=w\delta_D(w)\delta_D^2(w)\cdots \delta_D^{k-1}(w)$ and $I_0=\delta_D^{-1}(I_{\tilde{w}})$. By our discussion before, we have $(\delta_D^{-1}(I_{\tilde{w}}),w,I_{\tilde{w}})$ is an admissible triple. Set $I_j=\delta_D^{j-1}(I_{\tilde{w}})$, we have $(I_j,\delta_D^j(w),I_{j+1})$ are all admissible triples. Meanwhile, notice that $b=\underline{w_0w'}\cdot b'$ by our assumption above, we have $\tilde{O}(I,b)$ is affine by \cite[Proposition 7.26]{DM}. Thus $D\times \tilde{O}(I,b)$ is also affine. Now we consider 
\begin{equation*}
    \tilde{Y}':=\{(g,g^{(1)}U_{\delta_D^{-1}(I_{\tilde{w}})}, g^{(2)}U_{I_{\tilde{w}}},\cdots, g^{(k)}U_{\delta_D^{k-2}(I_{\tilde{w}})})\}\cap (D\times \tilde{O}(I,b)).
\end{equation*}
By above we have $\tilde{Y}'$ is a closed subvariety of $D\times \tilde{O}(I,b)$ and thus is affine. We define a map $\Upsilon$ from $ \tilde{Y}^D_{\delta_D^{-1}(I_{\tilde{w}}),I_{\tilde{w}}}(w)$ to $\tilde{Y}'$ given by
\begin{equation*}
    (g,g'U_{\delta_D^{-1}(I_{\tilde{w}})})\mapsto (g,g'U_{\delta_D^{-1}(I_{\tilde{w}})}, gg'U_{\delta_D^{-1}(I_{\tilde{w}})}d^{-1},\cdots, g^{k-1}g'U_{\delta_D^{-1}(I_{\tilde{w}})}d^{1-k}).
\end{equation*}
First notice that $gg'U_{\delta_D^{-1}(I_{\tilde{w}})}d^{-1}=gg'd^{-1}dU_{\delta_D^{-1}(I_{\tilde{w}})}d^{-1}=gg'd^{-1}U_{I_{\tilde{w}}}\in G\slash U_{I_{\tilde{w}}}$ and similar $g^{j}g'U_{I_{\tilde{w}}}d^{-j}\in G\slash U_{\delta_D^{j-1}(I_{\tilde{w}})}$ for any $j$. By the definition of $ \tilde{Y}^D_{\delta_D^{-1}(I_{\tilde{w}}),I_{\tilde{w}}}(w)$, we have $(g'U_{\delta_D^{-1}(I_{\tilde{w}})},gg'U_{\delta_D^{-1}(I_{\tilde{w}})}d^{-1})\sim (U_{\delta_D^{-1}(I_{\tilde{w}})},g'^{-1}gg'U_{\delta_D^{-1}(I_{\tilde{w}})}d^{-1})$, which lies in $\tilde{O}(\delta_D^{-1}(I_{\tilde{w}}),w)$. By similar method we see $\Upsilon$ is well-defined. Moreover, we have $\Upsilon$ is an isormorphism. This tells that $ \tilde{Y}^D_{\delta_D^{-1}(I_{\tilde{w}}),I_{\tilde{w}}}(w)$ is affine. Recall that $\pi_{\delta_D,w}$ is a principal $L^w_{\delta^{-1}(I_{\tilde{w}})}$-bundle where $L^w_{\delta^{-1}(I_{\tilde{w}})}$ is reductive, we have the variety $ Y^D_{\delta_D^{-1}(I_{\tilde{w}}),I_{\tilde{w}}}(w)$ is also affine.
\end{proof}

\section{$\CC$-small classes and orbit spaces}
In this section, we generalized the definition of $\CC$-small classes to general case using parabolic Lusztig varieties in \S 3.3 (non-twisted version). We also investigate orbit spaces of these varieties.

\subsection{$\CC$-small classes}
Let $\CC$ be any conjugacy class of $W$ and $w\in \CC$ be the natural projection of a good position braid representative of $\CC$. Let $I_w\subset S$ be as in \S 3.3. We have $w$ is a minimal double coset representative in $^{I_w}W^{I_w}$ and $O_{I_w}(w)$ is the $G$-orbit on $\CP_{I_w}\times \CP_{I_w}$ corresponding to $w$. Recall that the associated generalized Lusztig variety is
\begin{equation*}
    Y_{I_w}(w)=\{(g,P)\in G\times \CP_{I_w}\mid (P,gPg^{-1})\in O_{I_w}(w)\}.
\end{equation*}
There is a natural stratification $Y_{I_w}(w)=\underset{\CO\in [G]}{\bigsqcup} Y_{I_w}^\CO(w)$, where
\begin{equation*}
    Y_{I_w}^{\CO}(w)=\{(g,P)\in Y_{I_w}(w)\mid g\in \CO\}.
\end{equation*}
For any $\CO\in [G]$, the variety $Y_{I_w}^{\CO}(w)$ is a union of $G$-orbits in $Y_{I_w}^{\CO}$. By theorem \ref{3.7}, we know each $G$-orbit has dimension at least $\dim G-\dim L_{I_w}^w$. Then we have 
\begin{equation*}
    \dim Y_{I_w}^{\CO}\geq \dim G-\dim L_{I_w}^w=\dim G-\dim T^w-\lvert R^w \rvert.
\end{equation*}

\begin{definition} \label{4.1}
    A conjugacy class $\CO$ of $G$ is called $\CC$-small if
    \begin{equation*}
        \dim Y_{I_w}^{\CO}=\dim G-\dim T^w-\lvert R^w \rvert,
    \end{equation*}
    for $w$, which is the natural projection of some/any good position braid representative of $\CC$.
\end{definition}

\begin{corollary} \label{4.2}
    Let $\CC\in [W]$ and $w\in \CC$ be the natural projection of a good position braid representative of $\CC$. For any $g\in G$, the variety $\CP^w_g:=\{P\in \CP_{I_w} \mid (P,gPg^{-1})\in O_{I_w}(w)\}$ is smooth of pure dimension $l(w)$ if nonempty.
\end{corollary}
\begin{proof}
    Conisder the intersection of $O_{I_w}(w)$ and $\{(P,gPg^{-1})\mid P\in \CP_{I_w}\}$ in $\CP_{I_w}\times \CP_{I_w}$. It suffices to show that these two varieties intersects transversally if $\CP_g^w$ is nonempty. Indeed, we first notice that $\dim \{(P,gPg^{-1})\mid P\in \CP_{I_w}\}=\dim \CP_{I_w}$. Meanwhile, we have $\dim O_{I_w}(w)=\dim \CP_{I_w}+l(w)$ since $w$ is a minimal double coset representative. Then the dimension of the intersection will be $l(w)$ by transversality.

    Suppose that $\CP_g^w$ is nonempty, let $(P,gPg^{-1})$ be arbitrary in the above intersection. Let $\fkg,\fkp$ be the Lie algebras of $G$ and $P$. As in the proof of \cite[Lemma 1.1]{Lu0}, it suffices to show that $\ker(1-\Ad(g)\mid \fkg^* \to \fkg^*)\cap \fkp^\perp=0$. Set $b=w^d\in B^+(W)$. As in the proof of theorem \ref{3.7}(c), we have $(I_w,w,I_w)$ is an admissible triple. Consider a sequence $(\fkp_1,\fkp_2,\cdots,\fkp_{d+1})=(\fkp,\Ad(g)\fkp,\cdots,\Ad(g)^{d}\fkp)$ of subalgebras of $\fkg$ conjugate to $\fkp$. Here $\fkp_i$ is the Lie algebra of $P_i=g^{i-1}Pg^{1-i}$. We have $(P_i,P_{i+1})\in O_{I_w}(w)$ for all $i$. Let $x$ be arbitrary in $\ker(1-\Ad(g)\mid \fkg^* \to \fkg^*)\cap \fkp^\perp$. We must have $x\in \fkp_i^\perp$ for all $i$. Meanwhile, by the proof of \cite[Proposition 3.2]{D}, there is another decomposition $b=w_1w_2\cdots w_r$ with $w_1=w_0w'$ and each $w_i$ is a minimal double coset representative of $^{I_w}W^{I_w}$. Therefore, by \cite[\S 7.1]{DM} we have $x\in \fkp'^\perp$ where $\fkp'$ is the Lie algebra of $P'$ and $(P,P')\in O_{I_w}(w_0w')$. It is then clear that $\fkp^\perp \cap \fkp'^\perp=0$ and our statement is proved.
\end{proof}

Then we have the following result since the fiber of projection $Y_{I_w}^{\CO}\to \CO$ over any $g\in G$ is isomorphic to $\CP_g^w$.
\begin{corollary}
        Keep the notations in corollary \ref{4.2}. Let $\CO$ be any conjugacy class of $G$, then the variety $Y_{I_w}^{\CO}(w)$ is smooth of pure dimension $\dim \CO +l(w)$ if nonempty.
\end{corollary}

Now we recall that Lusztig defined a map $\Psi$ from unipotent orbits in $G$ to conjugacy classes of $W$ in \cite[Theorem 0.2]{Lu3}. Now for any unipotent orbit $\CO\in [G_u]$ and $\CC=\Psi(\CO)$,  we have
\begin{equation*}
    \dim \CO+l(w)+\lvert R^w \rvert =\dim G-\dim T^w,
\end{equation*}
by \cite[Proposition 3.6]{D}. In particular, if $Y_{I_w}^{\CO}(w)$ is nonempty, then $\CO$ is $\CC$-small. The following result is then direct.

\begin{proposition} \label{4.4}
      Let $\CO\in [G_u]$ be arbitrary and $\CC=\Psi(\CO)$. Let $w\in \CC$ be the natural projection of a good position braid representative of $\CC$. Then the $G$-action on $Y_{I_w}^{\CO}(w)$ has finitely many orbits if  nonempty.
\end{proposition}

\subsection{Orbit spaces}
In this section, we investigate the spaces of $G$-actions on $Y_{I_w}(w)$ and $\tilde{Y}_{I_w}(w)$.  Recall that $\tilde{Y}_{I_w}(w)=\{(g,g'U_{I_w})\in G \times G\slash U_{I_w}\mid g'^{-1}gg'\in U_{I_w}\dot{w}U_{I_w}\}$. Moreover, it is actually isomorphic to
\begin{equation*}
    \{(g,g'{}^wU_{I_w})\in G\times G\slash {}^wU_{I_w}\mid g'^{-1}gg'\in \dot{w}U_{I_w}\}.
\end{equation*}
Here $^wU_{I_w}=U_{I_w}\cap \dot{w}U_{I_w}\dot{w}^{-1}$, which corresponds to all roots in $R^+\setminus (R^w\cup \Inv(w^{-1}))$. 

We first prove a technical lemma. Let $w$ be the natural projection of a good position braid representative $b$ of $\CC$. Let $w'$ be the longest element in the standard parabolic subgroup $W'<W$ corresponding to $R^w$. Recall in \cite[\S 4.1]{D}, we have $b=\underline{w'}\cdot \underline{w'w}$ and the corresponding slice is defined as $S_{Br}(b)=U^b\dot{w}$ where $U^b=U^{w'}T^w\dot{w}'U^{w'w}\dot{w}'^{-1}$. As in \S 3.3, let $U_{(R^w)^+}$ (resp. $U_{(R^w)^-})$ be generated by all root subgroups corresponding to $(R^w)^+$ (resp. $(R^w)^-$). We also set $^{w^{-1}}U_{I_w}=U_{I_w}\cap \dot{w}^{-1}U_{I_w}\dot{w}$. Using results there, one can prove the following lemma.

\begin{lemma}\label{4.5}
    Let $w,b$ and $U^b$ be as above. There is a bijection
    \begin{equation*}
        \eta_w: {}^{w^{-1}}U_{I_w}\times U^b\to UT^wU_{(R^w)^-},
    \end{equation*}
    given by $(u,g)\mapsto ug\dot{w}u^{-1}\dot{w}^{-1}$.
\end{lemma}
\begin{proof}
    Let $h$ be arbitrary in $UT^wU_{(R^w)^-}$. Then we have $h\dot{w}\in US_{Br}(b)$ and is then contained in $S_{Br}(b)U_{I_w}$ by \cite[\S 4.2]{D}. Therefore, by the first part of \cite[Theorem 1.2]{D}, one can write 
    \begin{equation*}
        h\dot{w}=u_0g\dot{w}u_0^{-1},
    \end{equation*}
    for some $u_0\in U_{I_w}$ and $g\in U^b$. Now consider $u_0^{-1}h\in U_{I_w}U^b=U^b\cdot {}^wU_{I_w}$. Write $u_0^{-1}h=u_1u_2$ with $u_1\in U^b$ and $u_2\in {}^wU_{I_w}$. We then have $u_1u_2\dot{w}=g\dot{w}u_0^{-1}$. This is equivalent to
    \begin{equation*}
        u_1\dot{w}(\dot{w}^{-1}u_2\dot{w})=g\dot{w}u_0^{-1}.
    \end{equation*}
    Since $u_1\dot{w},g\dot{w}\in S_{Br}(b)$ and $\dot{w}^{-1}u_2\dot{w},u_0^{-1}\in U_{I_w}$, by \cite[Lemma 4.6]{D} we must have $u_1=g$ and $u_0^{-1}=\dot{w}^{-1}u_2\dot{w}\in {}^{w^{-1}}U_{I_w}$. This tells that $\eta_w$ is surjective. The injectivity of $\eta_w$ follows directly from \cite[Theorem 1.2]{D}.
\end{proof}

Now we may investigate the orbit spaces $G\backslash \tilde{Y}_{I_w}(w)$ and $G\backslash Y_{I_w}(w)$.

\begin{lemma}\label{4.6}
    Let $\tilde{\mathscr{O}}$ be a $G$-orbit in $\tilde{Y}_{I_w}(w)$. There is a unique $v\in U^w$ such that $(\dot{w}v,{}^wU_{I_w})\in \tilde{\mathscr{O}}$.
\end{lemma}
\begin{proof}
    For existence, we first notice that there exists $u\in U_{I_w}$ with $(\dot{w}u,{}^wU_{I_w})\in \tilde{\mathscr{O}}$. Then it suffices to find $v\in U^w$ such that $h\dot{w}uh^{-1}=\dot{w}v$ for some $h\in {}^wU_{I_w}$. This is equivalent to $u=\dot{w}^{-1}h^{-1}\dot{w}uh$. Now we set $\dot{w}^{-1}h^{-1}\dot{w}=h'$. It suffices to show that $u=h'v\dot{w}h'^{-1}\dot{w}^{-1}$ for some $h'\in \dot{w}^{-1}{}^wU_{I_w}\dot{w}={}^{w^{-1}}U_{I_w}$ and $v\in U^w$.
    
    Now by lemma \ref{4.5}, we have $u=h'v\dot{w}h'^{-1}\dot{w}^{-1}$  for some $h'\in {}^{w^{-1}}U_{I_w}$ and $v\in U^b$. Since $u\in U_{I_w}$, we must have $v\in U^w$. Indeed, write $v=v_1l$ for some $v_1\in U^w$ and $l\in L$ where $L=U_{(R^w)^+}T^wU_{(R^w)^-}$ is a reductive subgroup of $G$. Then $l=1$ by checking the corresponding root subgroups and the torus part of $h'v\dot{w}h'^{-1}\dot{w}^{-1}$. This proves the existence of $v$.

    As for uniqueness, suppose that $(\dot{w}v_1,{}^wU_{I_w})$ and $(\dot{w}v_2,{}^wU_{I_w})$ are both in $\tilde{\mathscr{O}}$ with $v_1,v_2
    \in U^w$. Then there exists $h\in {}^wU_{I_w}$ with $\dot{w}v_1=h\dot{w}v_2h^{-1}$. Again by lemma \ref{4.5}, we must have $h=1$ and $v_1=v_2$.
\end{proof}

From above result we know the intersection of the closed subvariety $\{(\dot{w}v,{}^wU_{I_w})\mid v\in U^w\}$ of $\tilde{Y}_{I_w}(w)$ with each $G$-orbit in $\tilde{Y}_{I_w}(w)$ is a singleton.

\begin{lemma}\label{4.7}
    The intersection of $\{(\dot{w}v,P_{I_w})\mid v\in U^w\}$ with any $G$-orbit in $Y_{I_w}(w)$ is a single $L_{I_w}^w$-orbit.
\end{lemma}
\begin{proof}
    Let $\mathscr{O}$ be any $G$-orbit in $Y_{I_w}(w)$ and $Z=\{(\dot{w}v,P_{I_w})\mid v\in U^w\}\cap \mathscr{O}$. Let $\pi_{\id,w}:\tilde{Y}_{I_w}(w)\to Y_{I_w}(w)$ be the covering map as defined in \S 3.3 where $\id$ is the trivial twist. Let $\tilde{\mathscr{O}}$ be the $G$-orbit in $\tilde{Y}_{I_w}(w)$ such that $\pi_{\id,w}(\tilde{\mathscr{O}})=\mathscr{O}$. By lemma \ref{4.6}, we have $Z$ is nonempty. Let $(\dot{w}v,P_{I_w})$ be arbitrary in $Z$ and $l\in L_{I_w}^w$, then $(l\dot{w}vl^{-1},P_{I_w})\in \mathscr{O}$. Notice that $l\dot{w}vl^{-1}=\dot{w}lvl^{-1}$ since $l\in L_{I_w}^w$. Moreover, since the unipotent parts of $L_{I_w}^w$ corresponds to the roots in $R^w$ and $v\in U^w$, we have $lvl^{-1}\in U^w$. Indeed, let $\alpha\in R^w$ and $\beta\in \Inv(w^{-1})$, then $m\alpha+n\beta\in R$ for some $m,n\in \BZ_{>0}$ implies $m\alpha+n\beta\in \Inv(w^{-1})$ by our proof of \cite[Lemma 4.5]{D}. Therefore, let $U_{\alpha}$ be the root subgroup of $\alpha$ and $u\in U_{\alpha}$, we must have $uvu^{-1}\in U^w$ by Chevalley commutator formula. This tells that $(l\dot{w}vl^{-1},P_{I_w})\in Z$ and $L_{I_w}^w$ acts on $Z$.

    Then we show $Z$ is a single $L_{I_w}^w$-orbit. Suppose that $(\dot{w}v_1,P_{I_w})$ and $(\dot{w}v_2,P_{I_w})$ are both in $Z$ where $v_1,v_2\in U^w$. Then $(\dot{w}v_1,{}^wU_{I_w})$ and $(\dot{w}v_2,{}^wU_{I_w})$ are both in $\tilde{\mathscr{O}}$. Thus $(\dot{w}v_1,{}^wU_{I_w})=(x\dot{w}v_2x^{-1},x{}^wU_{I_w})$ for some $x\in G$. Since $\pi_{\id,w}$ is a principal $L_{I_w}^w$-bundle, we must have $(\dot{w}v_1,l^{-1}{}^wU_{I_w})=(x\dot{w}v_2x^{-1},x{}^wU_{I_w})$ for some $l\in L_{I_w}^w$. By apply the action of $l$ and $x$ to two sides separately, we have $(l\dot{w}v_1l^{-1},{}^wU_{I_w})$ and $(\dot{w}v_2,{}^wU_{I_w})$ are in the same $G$-orbit. Since $L_{I_w}^w$ acts on $Z$, we have $l\dot{w}v_1l^{-1}=\dot{w}v_1'$ for some $v_1'\in U^w$. Then by lemma \ref{4.6} we must have $v_1'=v_2$ and thus $(\dot{w}v_1,P_{I_w})$ and $(\dot{w}v_2,P_{I_w})$ are in the same $L_{I_w}^w$-orbit. 
\end{proof}

Now we further assume that $\CC=\Psi(\CO)$ of some unipotent orbit $\CO\in [G_u]$. By proposition \ref{4.4}, we have the subset $Y_{I_w}^{\CO}(w)\subset Y_{I_w}(w)$ is the union of finitely many $G$-orbits or empty. Meanwhile, by the second part of \cite[Theorem 1.2]{D}, we know $S_{Br}(b)$ intersects $\CO$ transversally. This tells that $S_{Br}(b)\cap \CO$ is a finite nonempty set. Then by pulling back to $Y_{I_w}(w)$, we have $Y_{I_w}^{\CO}(w)$ is nonempty. Therefore, we have our second main result which summarizes these two sections.
\begin{proposition}\label{4.8}
    Let $\CC\in [W]$ and $w\in \CC$ be the natural projection of a good position braid representative of $\CC$. Then: 
 
    (1) The orbit space $G\backslash \tilde{Y}_{I_w}(w)$ is isomorphic to the affine space $U^w$;

    (2) The orbit space $G\backslash Y_{I_w}(w)$ is isomorphic to $L_{I_w}^w\backslash U^w$;

    If we further require $\CC=\Psi(\CO)$ of some unipotent orbit $\CO\in [G_u]$, then:

    (3) $\CO$ is the unique $\CC$-small unipotent conjugacy class in $G$.
\end{proposition}

\end{document}